%% file: main.tex
\documentclass{article}
\usepackage{arxiv}
\include{header}

\begin{document}

\title{Towards accelerated rates for distributed optimization over time-varying networks\thanks{The research of A. Rogozin was partially supported by RFBR 19-31-51001 and was partially done in Sirius (Sochi). The research of A. Gasnikov was partially supported by the Ministry of Science and Higher Education of the Russian Federation (Goszadaniye) 075-00337-20-03, project no. 0714-2020-0005.}
}

\author{
    Alexander Rogozin \\
    Moscow Institute of Physics and Technology \\
    Dolgoprudny, Russia \\
    \texttt{aleksandr.rogozin@phystech.edu} \\
    \And
    Vladislav Lukoshkin \\
    Skolkovo Institute of Science and Technology \\
    Moscow, Russia \\
    \texttt{lukoshkin@phystech.edu}
    \And
    Alexander Gasnikov \\
    Moscow Institute of Physics and Technology \\
    Dolgoprudny, Russia \\
    \texttt{gasnikov@yandex.ru}
    \And
    Dmitry Kovalev \\
    King Abdullah University of Science and Technology \\
    Thuwal, Saudi Arabia \\
    \texttt{dmitry.kovalev@kaust.edu.sa}
    \And
    Egor Shulgin \\
    King Abdullah University of Science and Technology \\
    Thuwal, Saudi Arabia \\
    \texttt{shulgin.ev@phystech.edu}
    \texttt{}
}

\renewcommand{\today}{\text{September 24, 2020}}

\maketitle

% \vspace{-1.cm}
% \begin{center}
%     Revised on November 7, 2022
% \end{center}

\begin{abstract}
	We study the problem of decentralized optimization with strongly convex smooth cost functions. This paper investigates accelerated algorithms under time-varying network constraints. In our approach, nodes run a multi-step gossip procedure after taking each gradient update, thus ensuring approximate consensus at each iteration. The outer cycle is based on accelerated Nesterov scheme. Both computation and communication complexities of our method have an optimal dependence on global function condition number $\kappa_g$. In particular, the algorithm reaches an optimal computation complexity $O(\sqrt{\kappa_g}\log(1/\varepsilon))$.
\end{abstract}

\section{Introduction}

In this work, we study a sum-type minimization problem
\begin{align}\label{eq:initial_problem}
f(x) = \frac{1}{n}\sum_{i=1}^n f_i(x) \to \min_{x\in\R^d}.
\end{align}
Convex functions $f_i$ are stored separately by nodes in communication network, which is represented by an undirected graph $\mathcal{G} = (V, E)$. This type of problems arise in distributed machine learning, drone or satellite networks, statistical inference \cite{nedic2017fast} and power system control \cite{ram2009distributed}. The computational agents over the network have access to their local $f_i$ and can communicate only with their neighbors, but still aim to minimize the global objective in \eqref{eq:initial_problem}.

The basic idea behind approach of this paper is to reformulate problem \eqref{eq:initial_problem} as a problem with linear constraints. Let us assign each agent in the network a personal copy of parameter vector $x_i$ and introduce
\begin{align*}
\mX = \begin{pmatrix} x_1^\top \\ \vdots \\ x_n^\top \end{pmatrix} \in \R^{n\times d},~
F(\mX) = \sum_{i=1}^n f_i(x_i).
\end{align*}
% Gradient of $F$ writes as
% \begin{align*}
% \nabla F(\mX) = \begin{pmatrix} \nabla f_1^\top(x_1) \\ \vdots \\ \nabla f_n^\top(x_n) \end{pmatrix}.	
% \end{align*}
Now we equivalently rewrite problem \eqref{eq:initial_problem} as
\begin{align}\label{eq:problem_linear_constraints}
F(\mX) = \sum_{i=1}^n f_i(x_i) \to \min_{x_1 = \ldots = x_n}.
\end{align}
This reformulation increases the number of variables, but induces additional constraints at the same time. Problem \eqref{eq:problem_linear_constraints} has the same optimal value as problem \eqref{eq:initial_problem}.

Let us denote the set of consensus constraints $\cset = \{x_1 = \ldots = x_n\}$. Also, for each $\mX\in\R^{d\times n}$ denote average of its columns $\ol x = \frac{1}{n}\sum_{i=1}^n x_i$ and introduce its projection onto constraint set.
\begin{align*}
\ol\mX = \frac{1}{n} \onevec\onevec^\top \mX = \Pi_{\cset}(\mX) = \begin{pmatrix} \ol x^\top \\ \vdots \\ \ol x^\top \end{pmatrix}.
\end{align*}
Note that $\cset$ is a linear subspace in $\R^{n\times d}$, and therefore projection operator $\Pi_\cset(\cdot)$ is linear.

Decentralized optimization methods aim at minimizing the objective function and maintaining consensus accuracy between nodes. The optimization part is performed by using gradient steps. At the same time, keeping every agent's parameter vector close to average over the nodes is done via communication steps. Alternating gradient and communication updates allows both to minimize the objective and control consensus constraint violation.

In centralized scenario, there exists a server which is able to communicate with every agent in the network. In particular, a common parameter vector is maintained at all of the nodes. However, in decentralized setting it is only possible to ensure that agent's vectors are approximately equal with desired accuracy. The algorithm studied in this paper runs a sequence of communication rounds after every optimization step. We refer to this series of communications as \textit{consensus subroutine}. Such information exchange allows to reach approximate consensus between nodes after each gradient update, while the accuracy is controlled by the number of communication rounds.

On the one hand, a method which employs a consensus subroutine after each gradient update mimics a centralized algorithm. The difference is that in presence of a master node all computational entities have access to a common variable, while in decentralized case consensus constraints are satisfied only with nonzero accuracy. On the other hand, consensus subroutine may be interpreted as an inexact projection onto the constraint set $\cset$. Every communication round is a step towards the projection. Therefore, our approach fits the inexact oracle framework, which has been studied in \cite{devolder2014first,devolder2013first}. We note that a similar approach to decentralized optimization is studied in \cite{jakovetic2014fast}.

We aim at building a first-order method with trajectory lying in neighborhood of $\cset$. A simple example would be GD with inexact projections.
\begin{align}\label{eq:example_gd}
\mX^{k+1} \approx \Pi_\cset(\mX^k - \gamma\nabla F(\mX^k)) = \ol\mX^k - \gamma \ol{\nabla F}(\mX^k),
\end{align}
where $\nabla F(\mX^k) = (\nabla f_1(x_1^k)\ldots \nabla f_n(x_n^k))^\top$ denotes the gradient of $F$.

Algorithm with update rule \ref{eq:example_gd} can be viewed as a gradient descent with inexact oracle. If the oracle was exact, the update rule would write as
\begin{align*}
\ol\mX^{k+1} = \ol\mX^k - \gamma\ol{\nabla F}(\ol\mX^k),
\end{align*}
thus making the method trajectory stay precisely in $\cset$. In this particular example, inexact gradient $\ol{\nabla F}(\mX^k)$ approximates exact gradient $\ol{\nabla F}(\ol\mX^k)$.

Throughout the paper, $\angles{\cdot, \cdot}$ denotes the inner product of vectors or matrices. Correspondingly, by $\norm{\cdot}$ we denote a $2$-norm for vectors or Frobenius norm for matrices.

\subsection{Related work}

A decentralized algorithm makes two types of steps: local updates and information exchange. The complexity of such methods depends on objective condition number $\kappa$ and a term $\chi$ representing graph connectivity.

Local steps may use gradient \cite{Nedic2017achieving,scaman2018optimal,Pu2018,Qu2017,shi2015extra,ye2020multi,li2018sharp} or sub-gradient \cite{Nedic2009} computations. In primal-only methods, the agents compute gradients of their local functions and alternate taking gradient steps and communication procedures. Under cheap communication costs, it may be beneficial to replace a single consensus iteration with a series of information exchange rounds. Such methods as MSDA \cite{scaman2017optimal}, D-NC \cite{Jakovetic} and Mudag \cite{ye2020multi} employ multi-step gossip procedures.

Typically, non-accelerated methods need $O(\kappa\chi\log(1/\eps))$ iterations to yield a solution with $\eps$-accuracy \cite{rogozin2019projected}. Nesterov acceleration may be employed to improve dependence on $\kappa$ or $\chi$ and obtain algorithms with $O(\sqrt{\kappa\chi}\log(1/\eps))$ complexity. In order to achieve this, one may distribute accelerated methods directly \cite{Qu2017,ye2020multi,li2018sharp,Jakovetic,dvinskikh2019decentralized} or use a Catalyst framework \cite{li2020revisiting}. Accelerated methods meet the lower complexity bounds for decentralized optimization \cite{hendrikx2020optimal,li2020optimal,scaman2017optimal}.

Consensus restrictions $x_1 = \ldots = x_n$ may be treated as linear constraints, thus allowing for a dual reformulation of problem \eqref{eq:initial_problem}. Dual-based methods include dual ascent and its accelerated variants \cite{scaman2017optimal,Wu2017,Zhang2017,uribe2020dual}. Primal-dual approaches like ADMM \cite{arjevani2020ideal,wei2012distributed} are also implementable in decentralized scenarios.

In \cite{scaman2018optimal}, the authors developed algorithms for non-smooth convex objectives and provided lower complexity bounds for non-smooth convex case, as well.

Time-varying networks open a new venue in research. Changing topology requires new approaches to decentralized methods and a more complicated theoretical analysis. The first method with provable geometric convergence was proposed in \cite{Nedic2017achieving}. Such primal algorithms as Push-Pull Gradient Method \cite{Pu2018} and DIGing \cite{Nedic2017achieving} are robust to network changes and have theoretical guarantees of convergence over time-varying graphs. Recently, a dual method for time-varying architectures was introduced in \cite{Maros2018}.

\subsection{Summary of contributions}

Our approach uses multi-step gossip averaging, and the analysis is based on the inexact oracle framework.

The proposed algorithm (Algorithm \ref{alg:decentralized_agd}) requires $O(\sqrt{\kappa_g} \chi \log^2(1/\eps))$, where $\kappa_g$ denotes the (global) condition number of $f$ and $\chi$ is a term characterizing graph connectivity, which is defined later in the paper. For a static graph, $\chi = 1 / \gamma$, where $\gamma$ denotes the normalized eigengap of communication matrix associated with the network. Our result has an accelerated rate on function condition number and is derived for time-varying networks.

Our complexity estimate includes the condition number $\kappa_g$, i.e. \textit{global} strong convexity and smoothness constants instead of \textit{local} ones. If the data among the nodes is strongly heterogeneous, it is possible that $\kappa_g \gg \kappa_l$, where $\kappa_l$ is the local condition number \cite{scaman2017optimal,hendrikx2020optimal,tang2019practicality}. Therefore, the method with complexity depending on $\kappa_g$ may perform significantly better. A recently proposed method Mudag \cite{ye2020multi} has a complexity depending on $\sqrt{\kappa_g}$, as well, but the method is designed for time-static graphs.

The lower bound for number of communications is $\Omega(\sqrt{\kappa_l \chi}\log(1/\eps))$ \cite{scaman2017optimal}. Our result is obtained for time-varying graphs and has a worse dependence on $\chi$. On the other hand, we derive a better dependence on condition number, i.e. we use $\kappa_g$ instead of $\kappa_l$. However, this by no means breaks the lower bounds. First, our result includes $\log^2(1/\eps)$ instead of $\log(1/\eps)$. Second, a function in \cite{scaman2017optimal} on which the lower bounds are attained has $\kappa_g\sim \kappa_l$. Namely, for the bad function it holds $\kappa_g\ge \kappa_l/16$ (see Appendix A.1 in \cite{scaman2017optimal} for details).
%	Paper \cite{scaman2017optimal} studies the lower complexity bounds for decentralized optimization over static graphs.
% 	\item We do not assume mixing matrices $\mW^k$ to be symmetric and doubly stochastic. In our approach, stochastity is sufficient. Our assumption (Assumption \ref{assum:mixing_matrix}) is more loose than ones typically used in literature \cite{scaman2017optimal, koloskova2020unified, Nedic2017achieving}. The pay for such generality is a non-accelerated rate on $\chi$.

\section{Inexact oracle framework}

In this section, we describe the inexact oracle construction for objective function $f$.

\subsection{Preliminaries}

Initially we recall the definition of $(\delta, L, \mu)$-oracle from \cite{devolder2013first}. Let $h(x)$ be a convex function defined on a convex set $Q\subseteq\R^m$. We say that $(h_{\delta,L,\mu}(x), s_{\delta,L,\mu}(x))$ is a $(\delta, L, \mu)$-model of $h(x)$ at point $x\in Q$ if for all $y\in Q$ it holds
\begin{align}\label{eq:inexact_oracle_def_devolder}
\frac{\mu}{2}\norm{y - x}^2 \le h(y) - \cbraces{h_{\delta,L,\mu}(x) + \angles{s_{\delta,L,\mu}(x), y - x}} \le \frac{L}{2} \norm{y - x}^2 + \delta.
\end{align}

%In \cite{stonyakin2020inexact}, the authors generalized the inexact oracle concept to Bregman geometry. Let $d(x)$ be a differentiable distance generating function and $V(x, y) = d(x) - d(y) - \angles{\nabla d(y), x - y}$ be Bregman divergence associated with $d$. Moreover, let $\hat\psi(x, y)$ satisfy
%\begin{align*}
%	\hat\psi(x, x) &= 0~ \forall x\in Q \\
%	\psi(x) &\ge \psi(z) + \angles{\nabla\psi(z), x - z} + mV(z, x)~\forall x, z\in Q, 
%\end{align*}
%where $\psi(x) = \hat\psi(x, y)$. Following the definitions in \cite{stonyakin2020inexact}, we call $(\hat h(y), \hat\psi(x, y))$ a $(\delta,L,\mu,m,V)$-model of function $h$ at a given point $x\in Q$ if for all $y\in Q$ it holds
%\begin{align*}
%	\mu V(y, x) \le h(y) - (\hat h(x) + \hat\psi(y, x)) \le LV(y, x) + \delta.
%\end{align*}
%This definition comes down to \eqref{eq:inexact_oracle_def_devolder} by taking $d(x) = \norm{x}^2/2$ and $\hat\psi(y, x) = \angles{\hat s, y - x}$.

\subsection{Inexact oracle for f}

Consider $\ol x, \ol y\in\R^{d}$ and define 
%$\ol\mX = (\ol x \ldots \ol x)^\top,~ \ol\mY = (\ol y \ldots \ol y)^\top \in \cset$.
\begin{align*}
\ol\mX = \begin{pmatrix} \ol x^\top \\ \vdots \\ \ol x^\top \end{pmatrix},~ \ol\mY = \begin{pmatrix} \ol y^\top \\ \vdots \\ \ol y^\top \end{pmatrix} \in \cset.
\end{align*}
Let $\mX\in\R^{n\times d}$ be such that $\Pi_\cset (\mX) = \ol\mX$ and $\norm{\ol\mX - \mX}^2\le \delta'$.

\begin{lemma}\label{lem:inexact_oracle}
	Define
	\begin{align*}
	\delta &= \frac{1}{2n}\cbraces{\frac{\Lmax^2}{\Lav} + \frac{2\Lmax^2}{\muav} + \Lmax - \mumin} \delta', \numberthis\label{eq:delta_inexact_oracle} \\
	f_{\delta, L, \mu}(\ol x, \mX) &= \frac{1}{n} \sbraces{F(\mX) + \angles{\nabla F(\mX), \ol\mX - \mX} + \frac{1}{2}\cbraces{\mumin - \frac{2\Lmax^2}{\muav}}\norm{\ol\mX - \mX}^2}, \\
	g_{\delta, L, \mu}(\ol x, \mX) &= \frac{1}{n}\sum_{i=1}^n \nabla f_i(x_i).
	\end{align*}
	Then $(f_{\delta,L,\mu}(\ol x, \mX), g_{\delta,L,\mu}(\ol x, \mX))$ is a $(\delta,2\Lav,\muav/2)$-model of $f$ at point $\ol x$, i.e.
	
	\begin{align*}
	\frac{\muav}{4}\norm{\ol y - \ol x}^2 \le f(\ol y) - f_{\delta, L, \mu}(\ol x, \mX) - \angles{g_{\delta, L, \mu}(\ol x, \mX), \ol y - \ol x} \le \Lav\norm{\ol y - \ol x}^2 + \delta.
	\end{align*}
\end{lemma}

\begin{proof}
	We aim at obtaining estimates for $F(\aY)$ similar to \eqref{eq:inexact_oracle_def_devolder}. First, we get a lower bound on $F(\ol\mY)$.
	\begin{align*}
	F(\ol\mY)
	&\ge F(\mX) + \sbraces{\angles{\nabla F(\mX), \ol\mX - \mX} + \frac{\mumin}{2}\norm{\mX - \ol\mX}^2} + \sbraces{\angles{\ol{\nabla F}(\ol\mX), \ol\mY - \ol\mX} + \frac{\muav}{2}\norm{\ol\mY - \ol\mX}^2} \\
	&= \sbraces{F(\mX) + \angles{\nabla F(\mX), \ol\mX - \mX} + \frac{\mumin}{2}\norm{\mX - \ol\mX}^2} + \angles{\ol{\nabla F}(\mX), \ol\mY - \ol\mX} \\
	&\qquad + \angles{\ol{\nabla F}(\ol\mX) - \ol{\nabla F}(\mX), \ol\mY - \ol\mX} + \frac{\muav}{2}\norm{\ol\mY - \ol\mX}^2. \numberthis\label{eq:inexact_oracle_lower_bound_1}
	\end{align*}
	Let us lower bound the term $\angles{\ol{\nabla F}(\ol\mX) - \ol{\nabla F}(\mX), \ol\mY - \ol\mX}$ using Young inequality $\angles{a, b}\le \frac{\norm{a}^2}{2p} + \frac{p}{2}\norm{b}^2,~ p > 0$.
	\begin{align*}
	\angles{\ol{\nabla F}(\ol\mX) - \ol{\nabla F}(\mX), \ol\mY - \ol\mX}
	&\ge -\frac{1}{2p}\norm{\ol{\nabla F}(\ol\mX) - \ol{\nabla F}(\mX)}^2 - \frac{p}{2}\norm{\ol\mY - \ol\mX}^2 \\ 
	&\ge -\frac{\Lmax^2}{2p}\norm{\ol\mX - \mX}^2 - \frac{p}{2}\norm{\ol\mY - \ol\mX}^2.
	\end{align*}
	Returning to \eqref{eq:inexact_oracle_lower_bound_1} and noting that $\angles{\ol{\nabla F}(\mX), \aY - \aX} = \angles{\nabla F(\mX), \aY - \aX}$, we get
	\begin{align*}
	F(\ol\mY) &\ge \sbraces{F(\mX) + \angles{\nabla F(\mX), \ol\mX - \mX} + \frac{1}{2}\cbraces{\mumin - \frac{\Lmax^2}{p}}\norm{\mX - \ol\mX}^2} \\
	&\qquad + \angles{\nabla F(\mX), \ol\mY - \ol\mX} + \frac{\muav - p}{2}\norm{\ol\mY - \ol\mX}^2 \numberthis\label{eq:inexact_oracle_lower_bound_2}
	\end{align*}
	
	Second, we get an upper estimate on $F(\ol\mY)$.
	\begin{align*}
	F(\ol\mY)
	&\le \sbraces{F(\mX) + \angles{\nabla F(\mX), \ol\mX - \mX} + \frac{\Lmax}{2}\norm{\ol\mX -\mX}^2} + \sbraces{\angles{\ol{\nabla F}(\ol\mX), \ol\mY - \ol\mX} + \frac{\Lav}{2}\norm{\ol\mY - \ol\mX}^2} \\
	&= \sbraces{F(\mX) + \angles{\nabla F(\mX), \ol\mX - \mX} + \frac{\Lmax}{2}\norm{\ol\mX -\mX}^2} + \angles{\ol{\nabla F}(\mX), \ol\mY - \ol\mX} \\
	&\qquad + {\angles{\ol{\nabla F}(\ol\mX) - \ol{\nabla F}(\mX), \ol\mY - \ol\mX}} + \frac{\Lav}{2}\norm{\ol\mY - \ol\mX}^2. \numberthis \label{eq:inexact_oracle_upper_bound_1}
	\end{align*}
	Analogously, we estimate the term $\angles{\ol{\nabla F}(\ol\mX) - \ol{\nabla F}(\mX), \ol\mY - \ol\mX}$ with Young inequality.
	\begin{align*}
	\angles{\ol{\nabla F}(\ol\mX) - \ol{\nabla F}(\mX), \ol\mY - \ol\mX}
	&\le \frac{1}{2q}\norm{\ol{\nabla F}(\ol\mX) - \ol{\nabla F}(\mX)}^2 + \frac{q}{2}\norm{\ol\mY - \ol\mX}^2 \\
	&\le \frac{\Lmax^2}{2q}\norm{\ol\mX - \mX}^2 + \frac{q}{2}\norm{\ol\mY - \ol\mX}^2,~ q > 0.
	\end{align*}
	Plugging it into \eqref{eq:inexact_oracle_upper_bound_1} and once again using $\angles{\ol{\nabla F}(\mX), \aY - \aX} = \angles{\nabla F(\mX), \aY - \aX}$ yields
	\begin{align*}
	F(\ol\mY)&\le \sbraces{F(\mX) + \angles{\nabla F(\mX), \ol\mX - \mX} + \frac{1}{2}\cbraces{\mumin - \frac{\Lmax^2}{p}}\norm{\mX - \ol\mX}^2} \\
	&\qquad + \angles{\nabla F(\mX), \ol\mY - \ol\mX} + \frac{\Lav + q}{2}\norm{\ol\mY - \ol\mX}^2 \\ 
	&\qquad + \frac{1}{2}\cbraces{\frac{\Lmax^2}{q} + \frac{\Lmax^2}{p} + \Lmax - \mumin}\norm{\mX - \ol\mX}^2 \numberthis\label{eq:inexact_oracle_upper_bound_2}
	\end{align*}
	Consequently, smoothness and strong convexity constants for inexact oracle are $L = \Lav + q,~ \mu = \muav - p$, respectively. Choosing $q$ and $p$ allows to control condition number $L/\mu$. Letting $q = \Lav,~ p = \muav / 2$ leads to
	\begin{subequations}\label{eq:L_mu_def}
		\begin{align}
		L &= 2\Lav, \\
		\mu &= \muav / 2.
		\end{align}
	\end{subequations}
	%Now we define
	%\begin{align*}
	%	\delta &= \frac{1}{2n}\cbraces{\frac{\Lmax^2}{\Lav} + \frac{2\Lmax^2}{\muav} + \Lmax - \mumin} \delta', \numberthis\label{eq:delta_inexact_oracle} \\
	%	f_{\delta, L, \mu}(\ol x) &= \frac{1}{n} \sbraces{F(\mX) + \angles{\nabla F(\mX), \ol\mX - \mX} + \frac{1}{2}\cbraces{\mumin - \frac{2\Lmax^2}{\muav}}\norm{\ol\mX - \mX}^2}, \\
	%	g_{\delta, L, \mu}(\ol x) &= \frac{1}{n}\sum_{i=1}^n \nabla f_i(x_i).
	%\end{align*}
	Noting that 
	\begin{align*}
	&F(\ol\mX) = nf(\ol x),~ F(\ol\mY) = nf(\ol y), \\
	&\norm{\ol\mY - \ol\mX}^2 = n\norm{\ol y - \ol x}^2, \\
	&\angles{\ol{\nabla F}(\mX), \ol\mY - \ol\mX} = n\angles{g_{\delta, L, \mu}(\ol x), \ol y - \ol x}
	\end{align*} 
	and combining \eqref{eq:inexact_oracle_lower_bound_2} and \eqref{eq:inexact_oracle_upper_bound_2} leads to
	\begin{align*}
	\frac{\mu}{2}\norm{\ol y - \ol x}^2 \le f(\ol y) - f_{\delta, L, \mu}(\ol x, \mX) - \angles{g_{\delta, L, \mu}(\ol x, \mX), \ol y - \ol x} \le \frac{L}{2}\norm{\ol y - \ol x}^2 + \delta, 
	\end{align*}
	which concludes the proof.
\end{proof}

\section{Algorithm and results}

We take Algorithm 2 from \cite{stonyakin2020inexact} as a basis for our method. The algorithm is designed for the inexact oracle model and achieves an accelerated rate.
\begin{algorithm}[H]
	\caption{Decentralized AGD with consensus subroutine}
	\label{alg:decentralized_agd}
	\begin{algorithmic}[1]
		\REQUIRE{Initial guess $\mX^0\in \cset$, constants $L, \mu > 0$, $\mU^0 = \mX^0$, $\alpha^0 = \mA^0 = 0$}
		\FOR{$k = 0, 1, 2,\ldots$}
		\STATE{Find $\alpha^{k+1}$ as the greater root of \\$(A^k + \alpha^{k+1})(1 + A^k \mu) = L(\alpha^{k+1})^2$}
		\STATE{$A^{k+1} = A^k + \alpha^{k+1}$}
		\STATE{$\ds \mY^{k+1} = \frac{\alpha^{k+1} \mU^k + A^k \mX^k}{A^{k+1}}$}
		\STATE{\label{alg_step:agd_step}$\mV^{k+1} = \frac{\alpha^{k+1}\mu\mY^{k+1} + (1 + A^k\mu)\mU^k}{1 + A^{k+1}\mu} - \frac{\alpha^{k+1}\nabla F(\mY^{k+1})}{1 + A^{k+1}\mu}$}
		\STATE{\label{alg_step:consensus_update}
			$
			\mU^{k+1} = \text{Consensus}(\mV^{k+1}, T^k)
			$
		}
		\STATE{$\ds \mX^{k+1} = \frac{\alpha^{k+1} \mU^{k+1} + A^k \mX^k}{A^{k+1}}$}
		\ENDFOR
	\end{algorithmic}
\end{algorithm}
\vspace{-0.5cm}
\begin{algorithm}[H]
	\caption{Consensus}
	\label{alg:consensus}
	\begin{algorithmic}
		\REQUIRE{Initial $\mX^0\in\cset$, number of iterations $T$.}
		\FOR{$t = 1, \ldots, T$}
		\STATE{$\mX^{t+1} = \mW^t \mX^t$}
		\ENDFOR
	\end{algorithmic}
\end{algorithm}

\subsection{Consensus}
We consider a sequence of non-directed communication graphs $\braces{\mathcal{G}^k = (V, E^k)}_{k=0}^\infty$ and a sequence of corresponding mixing matrices $\braces{\mW^k}_{k=0}^\infty$ associated with it. We impose the following 
\begin{assumption}\label{assum:mixing_matrix}
	Mixing matrix sequence $\braces{\mW^k}_{k=0}^\infty$ satisfies the following properties.
	\begin{itemize}
		\item (Decentralized property) If $(i, j)\notin E_k$, then $[\mW^k]_{ij} = 0$.
		\item (Double stochasticity) $\mW^k \onevec = \onevec,~ \onevec^\top\mW^k = \onevec^\top$.
		\item (Contraction property) There exist $\tau\in\Z_{++}$ and $\lambda\in(0, 1)$ such that for every $k\ge \tau - 1$ it holds
		\begin{align*}
		\norm{\mW_{\tau}^k \mX - \aX} \le (1 - \lambda)\norm{\mX - \aX}, 
		\end{align*}
		where $\mW_\tau^k = \mW^k \ldots \mW^{k-\tau+1}$.
	\end{itemize}
\end{assumption}

The contraction property in Assumption \ref{assum:mixing_matrix} generalizes several assumptions in the literature.
\begin{itemize}
    \item Time-static connected graph: $\mW^k = \mW$. In this classical case we have $\lambda = 1 - \sigma_2(\mW)$, where $\sigma_2(\mW)$ denotes the second largest singular value of $\mW$.
    \item Sequence of connected graphs: every $\mathcal{G}_k$ is connected. In this scenario $\lambda = 1 - \underset{k\ge 0}{\sup}~\sigma_2(\mW^k)$.
    \item $\tau$-connected graph sequence (i.e. for every $k\ge 0$ graph $\mathcal{G}^k_\tau = (V, E^k\cup E^{k+1}\cup\ldots\cup E^{k+\tau-1})$ is connected \cite{Nedic2017achieving}). For $\tau$-connected graph sequences it holds $1 - \lambda = \underset{k\ge 0}{\sup}~\sigma_{\max}(\mW_\tau^k - \frac{1}{n}\onevec\onevec^\top)$.
\end{itemize}
A stochastic variant of this contraction property is also studied in \cite{koloskova2020unified}.

%\begin{itemize}
%	\item If $(i, j)\notin E_k$, then $W_{ij}^k = 0$.
%	\item $\mW^k \onevec = \onevec,~ \onevec^\top \mW^k = \onevec^\top$.
%	\item $\ds \theta = \sup_{k\ge B - 1}\braces{\sigma_{\max}\sbraces{\mW_B^k - \frac{1}{n}\onevec\onevec^\top}} < 1$, where $\mW_B^k = \mW^k \mW^{k-1}\ldots \mW^{k-B+1}$.
%\end{itemize}
%The last assumption is related to $B$-connected graph sequence, i.e. a sequence for which a graph with edge set $\cup_{i=k}^{k+B-1} E_i$ is connected for every $k\ge 0$.

During every communication round, the agents exchange information according to the rule
\begin{align*}
x_i^{k+1} = w_{ii}^k + \sum_{(i, j)\in E^k} w_{ij}^k x_j^k.
\end{align*}
In matrix form, this update rule writes as $\mX^{k+1} = \mW^k \mX^k$. The contraction property in Assumption \ref{assum:mixing_matrix} is needed to ensure geometric convergence of Algorithm \ref{alg:consensus} to the average of nodes' initial vectors, i.e. to $\ol x^0$. In particular, the contraction property holds for $\tau$-connected graphs with Metropolis weights choice for $\mW^k$, i.e.
\begin{align*}
[\mW^k]_{ij} = 
\begin{cases}
1 / (1 + \max\{d^k_i, d^k_j\}) &\text{if }(i, j)\in E^k, \\
0 &\text{if } (i, j)\notin E^k, \\
1 - \ds\sum_{(i, m)\in E^k} [\mW^k]_{im} &\text{if } i = j, 
\end{cases}
\end{align*}
where $d^k_i$ denotes the degree of node $i$ in graph $\mathcal{G}^k$.

\subsection{Complexity result for Algorithm \ref{alg:decentralized_agd}}

This paper focuses on smooth strongly convex objectives. Our analysis is bounded to the following 
\begin{assumption}\label{assum:str_convex_smooth}
	For every $i = 1,\ldots, n$, function $f_i$ is differentiable, $\mu_i$-strongly convex and $L_i$-smooth ($\mu_i,~ L_i > 0$).
\end{assumption}
Under this assumption it holds 
\begin{itemize}
	\item (local constants) $F(X)$ is $\mumin$-strongly convex and $\Lmax$-smooth on $\R^{n\times d}$, where $\ds \mumin = \min_i\mu_i,~ \Lmax = \max_i L_i$.
	\item (global constants) $F(X)$ is $\muav$-strongly convex and $\Lav$-smooth on $\cset$, where $\muav = \frac{1}{n}\sum_{i=1}^n\mu_i,~ \Lav = \frac{1}{n}\sum_{i=1}^n L_i$.
\end{itemize}
The global conditioning may be significantly better than local (see i.e. \cite{scaman2017optimal} for details). Our analysis shows that performance of Algorithm \ref{alg:decentralized_agd} depends on global constants.

In the next theorem, we provide computation and communication complexities of Algorithm \ref{alg:decentralized_agd}.
\begin{theorem}\label{th:total_iterations_strongly_convex}
	Choose some $\eps > 0$ and set
	\begin{align*}
	T_k = T = \frac{\tau}{2\lambda}\log\frac{D}{\delta'},~ \delta' = \frac{n\eps}{32} \frac{\muav^{3/2}}{\Lav^{1/2} \Lmax^2}.
	\end{align*}
	Also define
	\begin{subequations}\label{eq:log_coefs_strongly_conv}
		\begin{align*}
		D_1 &= \frac{\Lmax}{\Lav^{1/2}\muav} \sbraces{ 8\sqrt{2}\Lmax\norm{\ol u^0 - x^*}\cbraces{\frac{\Lav}{\muav}}^{3/4} + \frac{4\sqrt{2}\norm{\nabla F(\mX^*)}}{\sqrt{n}}\cbraces{\frac{\Lav}{\muav}}^{1/4} } \\
		D_2 &= \frac{\Lmax}{\Lav^{1/2}\muav} \sbraces{ 3\sqrt{\muav} + 4\sqrt{2n}\cbraces{\frac{\Lav}{\muav}}^{1/4} }
		\end{align*}
	\end{subequations}
	Then Algorithm \ref{alg:decentralized_agd} requires
	\begin{align}\label{eq:agd_computational_complexity}
	N = 2\sqrt{\frac{\Lav}{\muav}} \log\cbraces{\frac{\norm{\ol u^0 - x^*}^2}{2\eps \Lav}}
	\end{align}
	gradient computations at each node and
	\begin{align}\label{eq:agd_communication_complexity}
	N_{tot} = N\cdot T =  2\sqrt{\frac{\Lav}{\muav}} \frac{\tau}{\lambda} \cdot \log\cbraces{\frac{2\Lav\norm{\ol u^0 - x^*}^2}{\eps}} \log\cbraces{\frac{D_1}{\sqrt\eps} + D_2}
	\end{align}
	communication steps to yield $\mX^N$ such that
	\begin{align*}
	&f(\ol x^N) - f(x^*)\le \eps,~ \norm{\mX^N - \aX^N}^2\le \delta'.
	\end{align*}
\end{theorem}
We provide the proof of Theorem \ref{th:total_iterations_strongly_convex} in Appendix \ref{app:total_iterations_strongly_convex}.

The number of gradient computations in \eqref{eq:agd_computational_complexity} reaches the lower bounds for non-distributed optimization up to a constant factor. Number of communication steps includes an additional term of $\tau/\lambda$, which characterizes graph connectivity.

\section{Numerical experiments}

We consider the logistic regression problem with L2 regularizer:
\begin{align*}
    f(x) = \frac{1}{n}\sum\limits_{i=1}^n \log\left(1 + \exp(-b_i\angles{a_i, x})\right) + \frac{\theta}{2}\norm{x}_2^2.
\end{align*}
Here $a_1,\ldots,a_n\in\R^d$ denote the data points, $b_1,\ldots,b_n\in\{-1, 1\}$ denote class labels and $\theta > 0$ is a penalty coefficient.

Also we run experiments on a least-squares task:
\begin{align*}
    f(x) = \frac{1}{2}\norm{\mA x - b}_2^2.
\end{align*}
The blocks of data matrix $\mA$ and vector $b$ are distributed among the agents in the network.

The simulations are run on LIBSVM datasets \cite{Chang2011}. We compare the performance of Algorithm \ref{alg:decentralized_agd} (DAccGD in legend of plots) to EXTRA \cite{shi2015extra}, DIGing \cite{Nedic2017achieving}, Mudag \cite{ye2020multi} and APM-C \cite{li2020decentralized}.

Logistic regression is carried out on a9a data-set, inner iterations are set to $T = 5$ for Mudag and DAccGD. Generation of random geometric graph goes on $20$ nodes.

\begin{figure}[H]
	\centering
	\includegraphics[width = \textwidth]{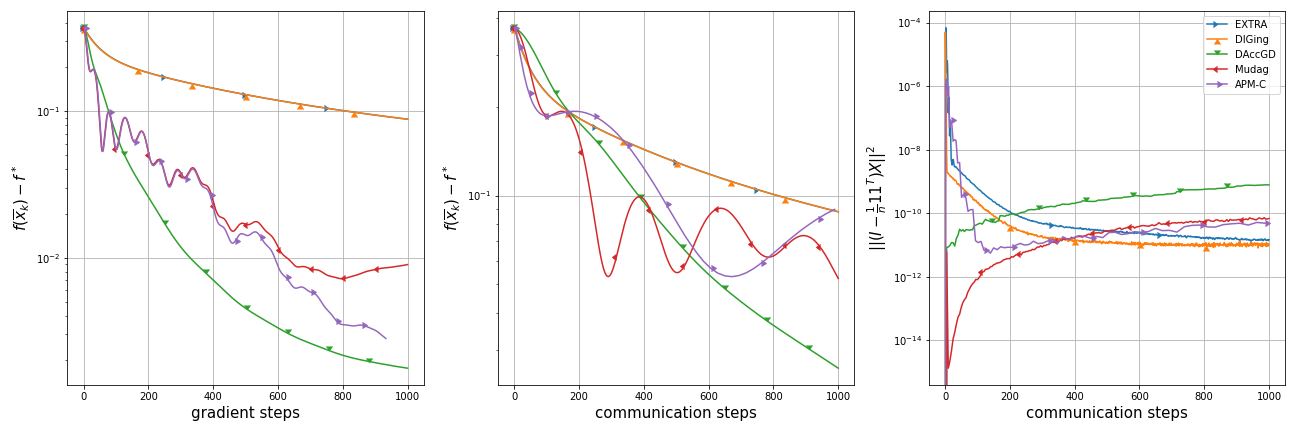}
	\caption{a9a (logistic regression), $100$ nodes}
\end{figure}

\begin{figure}[H]
	\centering
	\includegraphics[width = \textwidth]{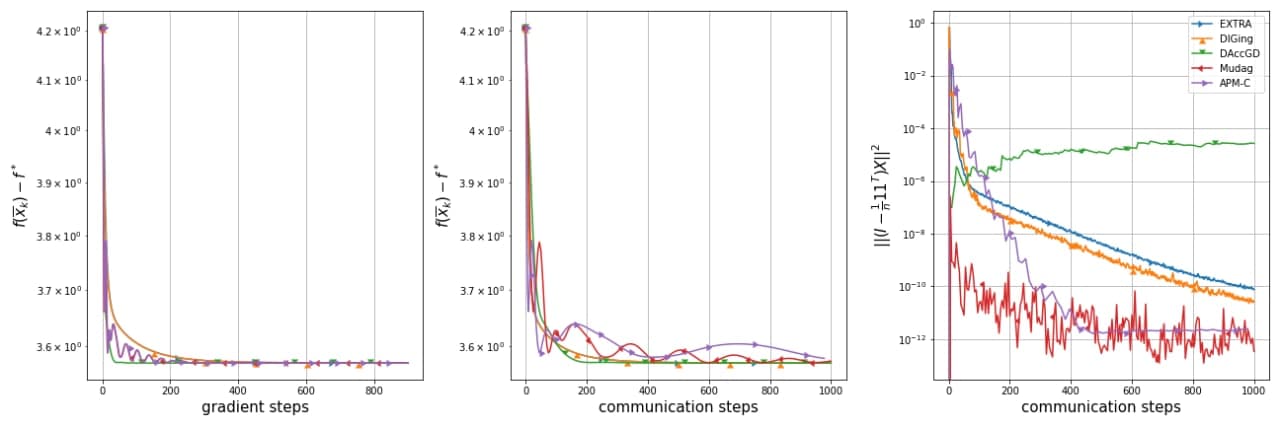}
	\caption{cadata (least squares), $20$ nodes}
\end{figure}

\section{Conclusion}

This paper studies an inexact oracle-based approach to decentralized optimization. The paper focuses on a specific case of strongly convex smooth functions, but the inexact oracle framework introduced in \cite{devolder2014first} is also applicable to non-strongly convex functions. The development of this framework in \cite{stonyakin2020inexact} also enables to generalize the results of this article to composite optimization problems and distributed algorithms for saddle-point problems and variational inequalities.

Another interesting application of inexact oracle approach lies in stochastic decentralized algorithms. Consider a class of $L$-smooth $\mu$-strongly convex objectives with gradient noise of each $f_i$ being upper-bounded by $\sigma^2$. For this class of problems, lower complexity bounds write as \cite{arjevani2015communication}
\begin{align*}
    O(\sqrt{L/\mu\chi}\log(1/\eps)) &\qquad\text{stochastic oracle calls per node} \\
    O\cbraces{\max\cbraces{\frac{\sigma^2}{m\mu\eps}, \sqrt{L/\mu}}\log(1/\eps)} &\qquad\text{communication steps}.
\end{align*}
At the moment, there exist methods which are optimal either in the number of oracle calls or in the number of communication steps. We believe that approach of this paper combined with a specific batch-size choice described in \cite{dvinskikh2020accelerated} allows to develop a decentralized algorithm reaching both optimal complexities.

\section{Acknowledgements}

The authors are grateful to Dmitriy Metelev and Fedor Stonyakin for their thoughtful proofreading of the paper.

\bibliographystyle{unsrt}
\bibliography{references}

\newpage
\appendix

\section{Proof of Theorem \ref{th:total_iterations_strongly_convex}}\label{app:total_iterations_strongly_convex}

\subsection{Gradient step}
Consider update rule
\begin{align}\label{eq:generalized_gradient_step}
Y = \argmin_{X\in\R^{n\times d}}\braces{\angles{P, X} + \frac{1}{2}\norm{X}^2}, \text{ where } P = \begin{pmatrix} p_1^\top \\ \vdots \\ p_n^\top \end{pmatrix}\in\R^{n\times d}.
\end{align}
By first-order optimality conditions $Y = -P$. Therefore $\ol y = -\ol p$ and
\begin{align}\label{eq:generalized_gradient_step_projection}
\ol y = \argmin_{x\in\R^d}\braces{\angles{\ol p, x} + \frac{1}{2}\norm{x}^2}.
\end{align}
This means that update rule \eqref{eq:generalized_gradient_step} in $\R^{n\times d}$ comes down to update rule \eqref{eq:generalized_gradient_step_projection} in $\R^{d}$.
In the case of gradient step, $P = \frac{1}{L}\nabla F(X^k) - X^k$, and step
\begin{align*}
X^{k+1} = X^k - \frac{1}{L}\nabla F(X^k)
\end{align*}
is comes down to $x^{k+1} = x^k - \frac{1}{L} \cbraces{\frac{1}{n}\sum_{i=1}^n \nabla f_i(x_i)}$.

Since update rule \ref{alg_step:agd_step} in Algorithm \ref{alg:decentralized_agd} has form \eqref{eq:generalized_gradient_step}, Algorithm \ref{alg:decentralized_agd} comes down to an iterative procedure in $\R^{d}$.
\begin{align*}
\ol y^{k+1} &= \frac{\alpha^{k+1}\ol u^k + A^k\ol x^k}{A^{k+1}} \\
\ol u^{k + 1} &= \argmin_{\ol z\in\R^d}\braces{\alpha^{k+1}\cbraces{\angles{\frac{1}{n}\sum_{i=1}^n \nabla f(y_i^{k+1}), \ol z - \ol y^{k+1}} + \frac{\mu}{2}\norm{\ol z - \ol y^{k+1}}^2} + \frac{1+ A^k\mu}{2}\norm{\ol z - \ol u^k}^2} \\
\ol x^{k+1} &= \frac{\alpha^{k+1}\ol u^{k+1} + A^k\ol x^k}{A^{k+1}}
\end{align*}

\subsection{Outer loop}

Initially we recall basic properties of coefficients $A^k$.
\begin{lemma}\label{lem:Ak_properties}
	For coefficients $A^k$, it holds
	\begin{align*}
	&1.~ A^N \ge \frac{1}{L}\cbraces{1 + \frac{1}{2}\sqrt\frac{\mu}{L}}^{2(N-1)} \\
	&2.~ \frac{\sum_{i=1}^k A^i}{A^k} \le {1 + \sqrt\frac{L}{\mu}}.
	\end{align*}
\end{lemma}
\begin{proof}
	For first statement see \cite{stonyakin2020inexact}, Lemma 3.7.
	
	Recall update rule for $A^k$:
	\begin{align}\label{eq:coef_sequence}
	1 + \mu A^k = \frac{L(\alpha^{k+1})^2}{A^{k+1}},~ A^{k} = \sum_{i=0}^k \alpha^{i},~ \alpha^0 = 0.
	\end{align}
	In \cite{devolder2013first} an almost similar sequence is studied:
	\begin{align}\label{eq:nesterov_coef_sequence}
	L + \mu B^k = \frac{L (\beta^{k+1})^2}{B^{k+1}},~ B^k = \sum_{i=0}^k \beta^i,~ \beta^0 = 1,
	\end{align}
	and for sequence $\braces{B^k}_{k=0}^\infty$ it is shown $\frac{\sum_{i=0}^k B^i}{B^k}\le 1 + \sqrt{L / \mu}$.
	Dividing \eqref{eq:nesterov_coef_sequence} by L yields
	\begin{align*}
	1 + \mu(B^k / L) = \frac{L (\beta^{k+1} / L)^2}{(B^{k+1} / L)},
	\end{align*}
	which means that update rule for $B^k / L$ is equivalent to \eqref{eq:coef_sequence}. Since $A^1 = 1 / L = B^0 / L$, it holds $A^{k+1} = B^k / L,~ k\ge 0$ and
	\begin{align*}
	\frac{\sum_{i=1}^k A^i}{A^k} = \frac{\sum_{i=0}^{k-1} B^i / L}{B^{k-1} / L} \le 1 + \sqrt\frac{L}{\mu}.
	\end{align*}
\end{proof}

\begin{lemma}\label{lem:inexact_agd_convergence}
	Provided that consensus accuracy is $\delta'$, i.e. $\norm{\mU^{j} - \ol\mU^j}^2\le \delta' \text{ for } j = 1, \ldots, k$, we have
	\begin{align*}
	f(\ol x^k) - f(x^*) &\le \frac{\norm{\ol u^0 - x^*}^2}{2A^k} + \frac{2\sum_{j=1}^{k} A^j\delta}{A^k} \\
	\norm{\ol u^k - x^*}^2 &\le \frac{\norm{\ol u^0 - x^*}^2}{1 + A^k\mu} + \frac{4\sum_{j=1}^{k} A^j\delta}{1 + A^k\mu}
	\end{align*}
	where $\delta$ is given in \eqref{eq:delta_inexact_oracle}.
\end{lemma}
\begin{proof}
	First, assuming that $\norm{\mU^{j} - \ol\mU^j}^2\le \delta'$, we show that $\mY^j, \mU^j, \mX^j$ lie in $\sqrt{\delta'}$-neighborhood of $\cset$ by induction. At $j=0$, we have $\norm{\mX^0 - \ol\mX^0} = \norm{\mU^0 - \ol\mU^0} = 0$. Using $A^{j+1} = A^j + \alpha^j$, we get an induction pass $j\to j+1$.
	\begin{align*}
	\norm{\mY^{j+1} - \ol\mY^{j+1}} &\le \frac{\alpha^{j+1}}{A^{j+1}}\norm{\mU^j - \ol\mU^j} + \frac{A^j}{A^{j+1}} \norm{\mX^j - \ol\mX^j}\le \sqrt{\delta'} \\
	\norm{\mX^{j+1} - \ol\mX^{j+1}} &\le \frac{\alpha^{j+1}}{A^{j+1}}\norm{\mU^{j+1} - \ol\mU^{j+1}} + \frac{A^j}{A^{j+1}} \norm{\mX^j - \ol\mX^j}\le \sqrt{\delta'}
	\end{align*}
	Therefore, $g(\ol y) = \frac{1}{n}\sum_{i=1}^n \nabla f(y_i)$ is a gradient from $(\delta, L, \mu)$-model of $f$, and the desired result directly follows from Theorem 3.4 in \cite{stonyakin2020inexact}.
\end{proof}

\subsection{Consensus subroutine iterations}
How many consensus iterations we need to reach accuracy $\delta'$? We formulate the corresponding result in the following
\begin{lemma}\label{lem:consensus_iters_strongly_convex}
	Let consensus accuracy be maintained at level $\delta'$, i.e. $\norm{\mU^j - \ol\mU^j}^2\le \delta' \text{ for } j = 1, \ldots, k$ and let Assumption \ref{assum:mixing_matrix} hold. Define
	\begin{align*}
	\sqrt D := \cbraces{\frac{2\Lmax}{\sqrt{L\mu}} + 1}\sqrt{\delta'} + \frac{\Lmax}{\mu} \sqrt{n} \cbraces{\norm{\ol u^0 - x^*}^2 + \frac{8{\delta'}}{\sqrt{L\mu}}}^{1/2} + \frac{2\norm{\nabla F(\mX^*)}}{\sqrt{L\mu}}
	\end{align*}
	
	Then it is sufficient to make $T_k = T = \frac{\tau}{2\lambda}\log\frac{D}{\delta'}$ consensus iterations in order to ensure $\delta'$-accuracy on step $k+1$, i.e. $\norm{\mU^{k+1} - \ol\mU^{k+1}}^2\le \delta'$.
\end{lemma}

\begin{proof}
	First, note that multiplication by mixing matrix preserves average, i.e. $\frac{1}{n}\onevec\onevec^\top\mX = \frac{1}{n}\onevec\onevec^\top \mW^k \mX~ \forall k\ge 0$. In particular, this implies $\ol\mU^{k+1} = \ol\mV^{k+1}$.
	
	By contraction property of mixing matrix sequence $\{\mW^k\}_{k=0}^\infty$ it holds
	\begin{align*}
	&\norm{\mU^{k+1} - \aU^{k+1}}^2 \le (1 - \lambda)^{2(T/\tau)} \norm{\mV^{k+1} - \aU^{k+1}}^2 \le e^{-2(T/\tau)\lambda} \norm{\mV^{k+1} - \aU^{k+1}}^2 \\
	%		&-2(T/\tau)\lambda + \log D \le \log\delta' \\
	%		&T \ge \frac{\tau}{2\lambda} \log\frac{D}{\delta'}.
	\end{align*}
	Assuming that $\norm{\mV^{k+1} - \aU^{k+1}}^2\le D$, we only need $T = \frac{\tau}{2\lambda}\log\frac{D}{\delta'}$ iterations to ensure $\norm{\mU^{k+1} - \aU^{k+1}}^2\le\delta'$. In the rest of the proof, we show that $\norm{\mV^{k+1} - \aU^{k+1}} = \norm{\mV^{k+1} - \aV^{k+1}} \le \sqrt D$.
	
	From optimality conditions for update rule \eqref{alg_step:agd_step} of Algorithm \ref{alg:decentralized_agd} it holds
	\begin{align*}
	\mV^{k+1} = \frac{\alpha^{k+1}\mu\mY^{k+1} + (1 + A^k\mu)\mU^k}{1 + A^{k+1}\mu} - \frac{\alpha^{k+1}}{1 + A^{k+1}\mu} \nabla F(\mY^{k+1})
	\end{align*}
	and thus
	\begin{align*}
	\norm{\mV^{k+1} - \ol\mV^{k+1}}
	&\le \frac{\alpha^{k+1}\mu \norm{\mY^{k+1} - \ol\mY^{k+1}}}{1 + A^{k+1}\mu} + \frac{(1 + A^k\mu)\norm{\mU^k - \ol\mU^k}}{1 + A^{k+1}\mu} + \frac{\alpha^{k+1}}{1 + A^{k+1}\mu} \norm{\nabla F(\mY^{k+1})} \\
	&\le \sqrt{\delta'} + \frac{\alpha^{k+1}}{1 + A^{k+1}\mu} \norm{\nabla F(\mY^{k+1})}.
	\end{align*}
	We estimate $\norm{\nabla F(\mY^{k+1})}$ using $\Lmax$-smoothness of $F$: 
	\begin{align*}
	\norm{\nabla F(\mY^{k+1})}
	&\le \norm{\nabla F(\mY^{k+1}) - \nabla F(\mX^*)} + \norm{\nabla F(\mX^*)} \\
	&\le \Lmax\underbrace{\norm{\mY^{k+1} - \ol\mY^{k+1}}}_{\le\sqrt{\delta'}} + \Lmax\underbrace{\norm{\ol\mY^{k+1} - \mX^*}}_{= \sqrt n \norm{\ol y^{k+1} - x^*}} + \norm{\nabla F(\mX^*)} \numberthis\label{eq:gradeint_y_upper_bound}
	\end{align*}
	where
	\begin{align*}
	\mX^* = \begin{pmatrix} (x^*)^\top \\ \vdots \\ (x^*)^\top \end{pmatrix},~
	x^* = \argmin_{x\in\R^d} f(x).
	\end{align*}
	It remains to estimate $\norm{\ol y^{k+1} - x^*}$.
	\begin{align*}
	\norm{\ol y^{k+1} - x^*}
	&\le \frac{\alpha^{k+1}}{A^{k+1}} \norm{\ol x^{k+1} - x^*} + \frac{A^{k}}{A^{k+1}}\norm{\ol u^{k+1} - x^*}
	\le \max\braces{\norm{\ol x^{k+1} - x^*}, \norm{\ol u^{k+1} - x^*}}
	\end{align*}
	
	By Lemma \ref{lem:inexact_agd_convergence} and strong convexity of $f$:
	\begin{align*}
	\norm{\ol x^{k+1} - x^*}^2 \le \frac{2}{\mu}\cbraces{f(\ol x^{k+1}) - f(x^*)} \le \frac{\norm{\ol u^0 - x^*}^2}{A^{k+1}\mu} + \frac{4\sum_{i=1}^{k+1} A^i \delta'}{A^{k+1}\mu}
	\end{align*}
	and therefore
	\begin{align*}
	\norm{\ol y^{k+1} - x^*}^2 
	&\le \max\braces{\frac{\norm{\ol u^0 - x^*}^2}{A^{k+1}\mu} + \frac{4\sum_{i=1}^{k+1} A^i \delta'}{A^{k+1}\mu},~ \frac{\norm{\ol u^0 - x^*}^2}{1 + A^{k+1}\mu} + \frac{4\sum_{i=1}^{k+1} A^i\delta'}{1 + A^{k+1}\mu}} \\
	&\overset{\circledOne}{\le} \frac{\norm{\ol u^0 - x^*}^2}{A^{k+1}\mu} + \frac{4}{\mu} \cbraces{1 + \sqrt\frac{L}{\mu}} \delta',
	\end{align*}
	where the last inequality holds by Lemma \ref{lem:Ak_properties}.
	
	Returning to \eqref{eq:gradeint_y_upper_bound}, we get
	\begin{align*}
	\norm{\nabla F(\mY^{k+1})}
	&\le \Lmax\sqrt{\delta'} + \Lmax\sqrt{n} \cbraces{\frac{\norm{\ol u^0 - x^*}^2}{A^{k+1}\mu} + \frac{4}{\mu} \cbraces{1 + \sqrt\frac{L}{\mu}} \delta'}^{1/2} + \norm{\nabla F(\mX^*)} \\
	&\le \Lmax\sqrt{\delta'} + \Lmax\sqrt{n} \cbraces{\frac{L}{\mu} \norm{\ol u^0 - x^*}^2 \cbraces{1 + \frac{1}{2}\sqrt\frac{\mu}{L}}^{-2k} + \frac{8L^{1/2}}{\mu^{3/2}}\delta'}^{1/2} + \norm{\nabla F(\mX^*)} \\
	&\le \Lmax\sqrt{\delta'} + \Lmax\sqrt{n} \cbraces{\frac{L}{\mu} \norm{\ol u^0 - x^*}^2 + \frac{8L^{1/2}}{\mu^{3/2}}{\delta'}}^{1/2} + \norm{\nabla F(\mX^*)}.
	\end{align*}
	For distance to consensus of $\mV^{k+1}$, it holds
	\begin{align*}
	\norm{\mV^{k+1} - \aV^{k+1}} \le \sqrt{\delta'} + \frac{\alpha^{k+1}}{1 + A^{k+1}\mu} \norm{\nabla F\cbraces{\mY^{k+1}}}
	\end{align*}
	We estimate coefficient by $\norm{\nabla F\cbraces{\mY^{k+1}}}$ using the definition of $\alpha^{k+1}$.
	\begin{align*}
	&1 + A^k\mu = \frac{L(\alpha^{k+1})^2}{A^k + \alpha^{k+1}} \\
	&L(\alpha^{k+1})^2 - (1 + A^k\mu)\alpha^{k+1} - (1 + A^k\mu)A^k = 0 \\
	&\alpha^{k+1} = \frac{1+ A^k\mu + \sqrt{(1 + A^k\mu)^2 + 4LA^k(1 + A^k\mu)}}{2L} \\
	&\frac{\alpha^{k+1}}{1 + A^{k+1}\mu} \le \frac{\alpha^{k+1}}{1 + A^k\mu} = \frac{1}{2L} \cbraces{1 + \sqrt{1 + 4\frac{LA^k}{1 + A^k\mu}}} \\
	&\qquad\le \frac{1}{2L}\cbraces{\sqrt{\frac{L}{\mu}} + \sqrt{\frac{L}{\mu} + 4\frac{L}{\mu}}} \le \frac{2}{\sqrt{L\mu}}
	\end{align*}
	Returning to $\mV^{k+1}$, we get
	\begin{align*}
	\norm{\mV^{k+1} - \aV^{k+1}}
	&\le \cbraces{\frac{2\Lmax}{\sqrt{L\mu}} + 1}\sqrt{\delta'} + \Lmax\sqrt{\frac{n}{L\mu}} \cbraces{\frac{L}{\mu} \norm{\ol u^0 - x^*}^2 + \frac{8L^{1/2}}{\mu^{3/2}}{\delta'}}^{1/2} + \frac{2\norm{\nabla F(\mX^*)}}{\sqrt{L\mu}} \\
	&= \cbraces{\frac{2\Lmax}{\sqrt{L\mu}} + 1}\sqrt{\delta'} + \frac{\Lmax}{\mu} \sqrt{n} \cbraces{\norm{\ol u^0 - x^*}^2 + \frac{8{\delta'}}{\sqrt{L\mu}}}^{1/2} + \frac{2\norm{\nabla F(\mX^*)}}{\sqrt{L\mu}} = \sqrt D
	\end{align*}
	which concludes the proof.
\end{proof}

\subsection{Putting the proof together}
Let us show that choice of number of subroutine iterations $T_k = T$ yields
\begin{align*}
&f(\ol x^k) - f(x^*)\le \frac{\norm{\ol u^0 - x^*}^2}{2A^k} + \frac{2\sum_{j=1}^k A^j\delta}{A^k}
\end{align*}
by induction. At $k=0$, we have $\norm{\mU^0 - \aU^0} = 0$ and by Lemma \ref{lem:inexact_agd_convergence} it holds
\begin{align*}
f(\ol x^1) - f(x^*) \le \frac{\norm{\ol u^0 - x^*}^2}{2A^1} + \frac{2A^1\delta}{A^1}.
\end{align*}
For induction pass, assume that $\norm{\mU^j - \aU^j}^2\le \delta'$ for $j = 0,\ldots, k$. By Lemma \ref{lem:consensus_iters_strongly_convex}, if we set $T_k = T$, then $\norm{\mU^{k+1} - \aU^{k+1}}^2\le \delta'$. Applying Lemma \ref{lem:inexact_agd_convergence} again, we get
\begin{align*}
&f(\ol x^{k+1}) - f(x^*)\le \frac{\norm{\ol u^0 - x^*}^2}{2A^{k+1}} + \frac{2\sum_{j=1}^{k+1} A^j\delta}{A^{k+1}}
\end{align*}	

Recalling a bound on $A^k$ from Lemma \ref{lem:Ak_properties} gives
\begin{align*}
f(\ol x^N) - f(x^*)
&\le \frac{L\norm{\ol u^0 - x^*}^2}{2}\cbraces{1 + \frac{1}{2}\sqrt\frac{\mu}{L}}^{-2(N-1)} + 2\cbraces{1 + \sqrt\frac{L}{\mu}}\delta \\
&\overset{\circledOne}{=} \Lav\norm{\ol u^0 - x^*}^2\cbraces{1 + \frac{1}{4}\sqrt\frac{\muav}{\Lav}}^{-2(N-1)} + 2\cbraces{1 + 2\sqrt\frac{\Lav}{\muav}}\delta
\end{align*}
Here in $\circledOne$ we used the definition of $L, \mu$ in \eqref{eq:L_mu_def}: $L = 2\Lav,~ \mu = \frac{\muav}{2}$.
%	\begin{align*}
%	\boxed{
%		q = \Lav,~ p = \frac{\muav}{2}\longrightarrow~ L = 2\Lav,~ \mu = \frac{\muav}{2}.
%	}
%	\end{align*}
For $\eps$-accuracy:
\begin{align*}
&\Lav\norm{\ol u^0 - x^*}^2\cbraces{1 + \frac{1}{4}\sqrt\frac{\muav}{\Lav}}^{-2(N-1)} \le \frac{\eps}{2} \longrightarrow~ N\ge 1 + \sqrt{\frac{\Lav}{\muav}} \log\cbraces{\frac{2\Lav\norm{\ol u^0 - x^*}^2}{\eps}} \\
&2\cbraces{1 + 2\sqrt\frac{\Lav}{\muav}}\delta \le \frac{\eps}{2} \longrightarrow~ \delta' \le \frac{n\eps}{2}\cbraces{1 + 2\sqrt\frac{\Lav}{\muav}}^{-1} \cbraces{\Lmax^2 \cbraces{\frac{1}{\Lav} + \frac{2}{\muav}} + \Lmax - \mumin}^{-1}
\end{align*}
It is sufficient to choose
\begin{align*}
&N = 2\sqrt\frac{\Lav}{\muav} \log\cbraces{\frac{2\Lav\norm{\ol u^0 - x^*}^2}{\eps}} \\
&\delta' = \frac{n\eps}{2} \cdot \frac{1}{2} \cdot \frac{1}{2}\sqrt\frac{\muav}{\Lav} \cdot \cbraces{4\frac{\Lmax^2}{\muav}}^{-1}
= \frac{n\eps}{32} \frac{\muav^{3/2}}{\Lav^{1/2} \Lmax^2}
\end{align*}
Let us estimate the term $\frac{D}{\delta'}$ under $\log$.
\begin{align*}
\sqrt\frac{D}{\delta'} 
&= \cbraces{\frac{2\Lmax}{\sqrt{L\mu}} + 1} + \frac{\Lmax}{\mu} \sqrt{n} \cbraces{\frac{\norm{\ol u^0 - x^*}^2}{\delta'} + \frac{8}{\sqrt{L\mu}}}^{1/2} + \frac{2\norm{\nabla F(\mX^*)}}{\sqrt{L\mu}\sqrt{\delta'}} \\
&\le \frac{3\Lmax}{\sqrt{\Lav\muav}} + \frac{2\Lmax\sqrt n}{\muav} \cbraces{ \sqrt{\frac{\norm{\ol u^0 - x^*}^2}{n\eps}\cdot \frac{32\Lav^{1/2}\Lmax^2}{\muav^{3/2}}} + \sqrt{\frac{8}{\sqrt{\Lav\muav}}} } + \frac{2\norm{\nabla F(\mX^*)}}{\sqrt{\Lav\muav}}\cdot \frac{\sqrt{32}}{\sqrt{n\eps}} \frac{\Lav^{1/4}\Lmax}{\muav^{3/4}} \\
&= \frac{3\Lmax}{\sqrt{\Lav\muav}} + \frac{8\sqrt{2} \Lmax^2\Lav^{1/4} \norm{\ol u^0 - x^*}}{\muav^{7/4}\sqrt{\eps}} + \frac{4\sqrt{2}\Lmax\sqrt{n}}{\Lav^{1/4}\muav^{5/4}} + \frac{4\sqrt{2}\norm{\nabla F(\mX^*)}\cdot \Lmax}{\Lav^{1/4}\muav^{5/4}\sqrt{n\eps}} \\
&= \frac{\Lmax}{\Lav^{1/2}\muav} \sbraces{ 3\sqrt\muav + \frac{8\sqrt{2}\Lmax\norm{\ol u^0 - x^*}}{\sqrt\eps}\cbraces{\frac{\Lav}{\muav}}^{3/4} + 4\sqrt{2n}\cbraces{\frac{\Lav}{\muav}}^{1/4} + \frac{4\sqrt{2}\norm{\nabla F(\mX^*)}}{\sqrt{n\eps}}\cbraces{\frac{\Lav}{\muav}}^{1/4} } \\
&= \frac{D_1}{\sqrt\eps} + D_2,
\end{align*}
where $D_1, D_2$ are defined in \eqref{eq:log_coefs_strongly_conv}.
Finally, the total number of iterations is
\begin{align*}
N_{\text{tot}} 
&= N\cdot T = 2\sqrt\frac{\Lav}{\muav} \log\cbraces{\frac{2\Lav\norm{\ol u^0 - x^*}^2}{\eps}}\cdot \frac{\tau}{2\lambda}\cdot 2\log\sqrt\frac{D}{\delta'} \\
&= 2\sqrt{\frac{\Lav}{\muav}} \frac{\tau}{\lambda} \cdot \log\cbraces{\frac{2\Lav\norm{\ol u^0 - x^*}^2}{\eps}} \log\cbraces{\frac{D_1}{\sqrt\eps} + D_2}.
\end{align*}

\end{document}

%% file: header.tex
\usepackage[utf8]{inputenc} % allow utf-8 input
\usepackage[T1]{fontenc}    % use 8-bit T1 fonts
\usepackage{hyperref}       % hyperlinks
\usepackage{url}            % simple URL typesetting
\usepackage{booktabs}       % professional-quality tables
\usepackage{amsfonts}       % blackboard math symbols
\usepackage{nicefrac}       % compact symbols for 1/2, etc.
\usepackage{lipsum}
\usepackage[english]{babel}

\usepackage{amsmath,amssymb,amsthm}
\usepackage{mathtools}
\usepackage{algorithm, algorithmic}
\usepackage{pifont}
\usepackage{cases}
\usepackage{subcaption,graphicx}
\usepackage{stackengine}    % circled symbols

\newtheorem{theorem}{Theorem}[section]

\newtheorem{lemma}[theorem]{Lemma}
\newtheorem{assumption}[theorem]{Assumption}

\newcommand{\eps}{\varepsilon}

\newcommand{\ol}{\overline}
\newcommand{\onevec}{\mathbf{1}_n}
\newcommand{\cset}{\mathcal{C}}

\newcommand{\circledOne}{\text{\ding{172}}}

\renewcommand{\le}{\leqslant}
\renewcommand{\ge}{\geqslant}

\newcommand{\numberthis}{\addtocounter{equation}{1}\tag{\theequation}}

\DeclareMathOperator*{\argmin}{arg\,min}

\newcommand{\R}{\mathbb{R}}
\newcommand{\Z}{\mathbb{Z}}

\newcommand{\mA}{{\bf A}}

\newcommand{\mU}{{\bf U}}
\newcommand{\mV}{{\bf V}}
\newcommand{\mW}{{\bf W}}
\newcommand{\mX}{{\bf X}}
\newcommand{\mY}{{\bf Y}}

\newcommand{\aX}{\overline{\bf X}}
\newcommand{\aY}{\overline{\bf Y}}

\newcommand{\aU}{\overline{\bf U}}
\newcommand{\aV}{\overline{\bf V}}

\newcommand{\muav}{\mu_{g}}
\newcommand{\Lav}{L_{g}}
\newcommand{\Lmax}{L_{l}}
\newcommand{\mumin}{\mu_{l}}

\newcommand{\ds}{\displaystyle}
\newcommand{\norm}[1]{\left\| #1 \right\|}

\newcommand{\angles}[1]{\left\langle #1 \right\rangle}
\newcommand{\cbraces}[1]{\left( #1 \right)}
\newcommand{\sbraces}[1]{\left[ #1 \right]}
\newcommand{\braces}[1]{\left\{ #1 \right\}}